\documentclass[12pt]{amsart}
\usepackage{amsmath,amssymb,amsthm}
\usepackage[latin1]{inputenc}
\usepackage{version,tabularx,multicol}
\usepackage{graphicx,float}
\usepackage[usenames]{color}

\newcommand{\reff}[1]{(\ref{#1})}

\theoremstyle{plain}
\newtheorem{theo}{Theorem}[section]

\newtheorem{cor}[theo]{Corollary}
\newtheorem{prop}[theo]{Proposition}
\newtheorem{proposition}[theo]{Proposition}
\newtheorem{lem}[theo]{Lemma}

\theoremstyle{remark}
\newtheorem{rem}[theo]{Remark}

\def\C{\ensuremath{\mathcal{C}}}

\def\D{\mathcal{D}}

\def\E{\ensuremath{\mathbb{E}}}

\def\F{\ensuremath{\mathcal{F}}}

\def\G{\ensuremath{\mathcal{G}}}

\def\M{\ensuremath{\mathcal{M}}}

\def\N{\ensuremath{\mathbb{N}}}

\def\p{\ensuremath{\mathcal{P}}}

\newcommand{\rP}{{\rm P}}

\def\P{\ensuremath{\mathbb{P}}}

\def\R{\ensuremath{\mathbb{R}}}

\def\equal{\ensuremath{\mathrel{\mathop:}=}}
\def\to{\rightarrow}

\def\equal{\ensuremath{\mathrel{\mathop:}=}}

\newcommand{\ind}{{\bf 1}}

\newcommand{\expp}[1]{\mathop {\mathrm{e}^{ #1}}}

\begin{document}

\title{Change of measure in the lookdown particle system}

\date{\today}

\author{Olivier H\' enard} 

\address{
Olivier H\' enard,
Universit\' e Paris-Est, CERMICS, 6-8
av. Blaise Pascal, 
  Champs-sur-Marne, 77455 Marne La Valle, France.}

\email{henardo@cermics.enpc.fr}

\thanks{This work is partially supported by the ``Agence Nationale de
  la Recherche'', ANR-08-BLAN-0190.}


\begin{abstract}
We perform various changes of measure in the lookdown particle system of Donnelly and Kurtz.
The first example is a product type h-transform related to conditioning a Generalized Fleming Viot process without mutation on coexistence of some genetic types in remote time. We give a pathwise construction of this h-transform by just ``forgetting'' some reproduction events in the lookdown particle system. We also provide an intertwining relationship for the Wright Fisher diffusion and explicit the associated pathwise decomposition.
The second example, called the linear or additive h-transform, concerns a wider class of measure valued processes with spatial motion. Applications include: -a simple description of the additive h-transform of the Generalized Fleming Viot process, which confirms a suggestion of Overbeck for the usual Fleming Viot process -an immortal particle representation for the additive h-transform of the Dawson Watanabe process.
\end{abstract}





\maketitle


\section{Introduction}

Measure valued processes are usually defined as rescaled limit of particle systems.
At the limit, the particle picture is lost. It is nevertheless often useful to keep track of the particles in the limiting process. First attempts to do that were concerned with a single particle, the most persistent one: this generated the so called spinal decompositions of superprocesses, see Roelly and Rouault \cite{RR89}, Evans \cite{EV93} and Overbeck \cite{OV94}.
Second attempts deal with many particles, still the most persistent, as the infinite lineages of a supercritical superprocess, see Evans and O'Connell \cite{EO94}.
Most interesting is to keep track of all the particles; this can be achieved by the following trick: ordering the particles by persistence (and giving them a label called the ``level'' accordingly) allows one to keep a particle representation of the full system \textit{after taking the limit}; the measure valued process is then represented by (a multiple of) the de Finetti measure of the exchangeable sequence formed by the types of the particles. This program was realized by Donnelly and Kurtz \cite{DO99} with the construction of the lookdown particle system.

Our aim in this article is to explain that some transformations of the law of measure valued processes, which belong to the class of $h$-transforms, admit a simple interpretation when considered from the lookdown particle system point of view. 

We recall Doob $h$-transform refers to the following operation: given a transition kernel $p_{t}(x,dy)$ of a Markov process and a positive space time harmonic function $H(t,y)$ for this kernel, meaning that: $$\int H(t,y)\; p_{t}(x,dy)=H(0,x)$$ for every $x$ and $t$, a new transition kernel is defined by $p_{t}(x,dy) H(t,y)/ H(0,x)$, and the associated Markov process is called an $h$-transform.  Working with measure valued processes, we may choose $H(t,y)$ to be a linear functional of the measure $y$, in which case the $h$-transform is called additive. 

Our contribution is the following: we observe that the Radon Nikodym derivative $H(t,y)/H(0,x)$ may be simply interpreted in term of the lookdown particle system in the following two cases: 
\begin{itemize}
 \item For a probability measure valued process on $\{1 \ldots K'\}$ called the Generalized Fleming Viot process without mutation, the choice $H(t,y)= \expp{r_K t} \prod_{i=1}^{K} y(\{i\})$ for   $1 \leq K \leq K'$, $y$ a probability measure on $\{1, \ldots, K'\}$ and $r_K$ a non negative constant chosen so that $H$ is harmonic, amounts to allocate the $K$ first types to the $K$ first particles. The corresponding $h$-transform is the process conditioned on coexistence of the $K$ first types in remote time, and the associated lookdown particle system is obtained by just ``forgetting'' some reproduction events in the original particle system. This way, we give a genealogy to the Wright Fisher diffusion conditioned on coexistence.
We also take the opportunity to present an intertwining relationship for the Wright Fisher diffusion and explicit the associated pathwise decomposition. This adds another decomposition to the striking one of Swart, see \cite{SW11}.
\item For a more general measure valued process on a Polish space $E$ (incorporating mutation and non constant population size), the choice $H(t,y)= \int y(du) h(t,u)$ for a suitable function $h(t,u):\R^+ \times E \to \R^+$ of the underlying mutation process and $y$ a finite measure on $E$, amounts to force the first level particle to move like an $h$-transform of the underlying spatial motion (or mutation process), and to bias the total mass process.
This confirms a suggestion of Overbeck about the additive $h$-transform of Fleming Viot processes, see \cite{DA92} p. 183. 
This also relates in the branching setting to decompositions of the additive $h$-transforms of superprocesses found by the same author \cite{OV94} using Palm measures. 
\end{itemize}
Our two examples, although similar, are independent: the first one may not be reduced to the second one, and vice versa.
We stress on the change of filtration technique, learnt in Hardy and Harris \cite{HH06}, which allows us to give simple proofs of the main results. 

We first recall in Section 2.1 the lookdown construction of \cite{DO99} in the case of the Generalized Fleming Viot process without mutation in finite state space. We look in Section 2.2 at the aforementioned product-type $h$-transform, and prove in Section 2.3 that it may be interpreted as the process conditioned on coexistence of some genetic types.
In Section 2.4, we compute the generator of the conditioned process in case the finite state space is composed of only two types, and recognize it as the generator of a Generalized Fleming Viot process with immigration. This Section also contains the statement and the interpretation of the intertwining relationship.
Section 3.1 starts with the introduction of a lookdown construction allowing for mutation and non constant population (also extracted from \cite{DO99}). We then present in Section 3.2 the additive $h$-transform of the associated measure valued process. Section 3.3 is concerned with applications in two classical cases: Dawson Watanabe processes and 
Generalized Fleming Viot processes.

\section{A product type $h$-transform}
\label{section2}

\subsection{The construction of the Generalized Fleming Viot Process without mutation.}

\label{spatial}

Donnelly and Kurtz  introduced  in \cite{DO99} a population model evolving in continuous time. We present in this Section the particular case of the Generalized Fleming Viot (GFV) process without mutation in which we are interested. A more general framework will be introduced in Section 3. 

We denote by $\p_{\infty}$ the space of partitions of the set of integers $\N=\{1, 2, 3, \ldots\}$. 
We assume $c \geq 0$. We define $\mu^k$ the measure on $\p_{\infty}$ assigning mass one to partitions with a unique non trivial block consisting of two different integers, and call it the Kingman measure. We let $\nu$ be a measure on $(0,1]$ satisfying $\int_{(0,1]} x^2 \nu(dx) < \infty$. We denote by $dt$ the Lebesgue measure on $\R_{+}$, and by $\rho_{x}$ the law of the exchangeable partition of $\N$ with a unique non trivial block with asymptotic frequency $x$: If $(U_i)_{i \in \N}$ is a sequence of independent Bernoulli random variables with parameter $x$, then the partition $\pi$ whose unique non trivial block contains the integers $i$ such that $U_i=1$ has law $\rho_x$. Finally, we define $N(dt,d\pi)$ the Poisson point measure on $\R_{+}\times \p_{\infty}$ with intensity $$dt \times \mu(d\pi) \equal  dt \times \left(c \mu^k(d\pi)+ \int_{(0,1]} \nu(dx) \rho_{x}(d\pi)\right).$$

Let $R_0$ be a random probability measure on the finite state space $E= \{1, 2, \ldots, K' \}$ for $K' \geq 2$. Assume $R_0$ is independent of $N$. Conditionally on $(R_{0},N)$, we define the lookdown particle system $X=(X_t(n), t \geq 0, n \in \N)$, and we call $X_{t}(n)$ the type of the particle at level $n \in \N$ at time $t \geq 0$:
\begin{itemize}
\item The initial state $(X_0(n), n \in \N)$ is an exchangeable sequence valued in $E$ with de Finetti's measure $R_{0}$: That means that, conditionally on $R_0$, $(X_0(n), n \in \N)$ form a sequence of independent random variables with law $R_0$.
\item At each atom $(t,\pi)$ of $N$, we associate a \textit{reproduction event} as follows: let $j_1 < j_2< \ldots$ be the elements of the unique block of the partition $\pi$ which is not a singleton (either it is a doubleton or an infinite set). The individuals $j_1<j_2 < \ldots$ at time $t$ are declared to be the children of the individual $j_1$ at time $t-$, and receive the type of the parent $j_1$, whereas the types of all the other individuals are shifted upwards accordingly, keeping the order they had before the birth event: for each integer $\ell$, $X_t(j_\ell)=X_{t-}(j_1)$ and for each $k \notin \{j_{\ell},\ell \in \N\}$, $X_t(k)=X_{t-}(k-\# J_k)$ with $J_k \equal \{\ell >1, j_{\ell} \leq k\}$ and $\# J_k$ the cardinal of the set $J_k$. 
\item For each $n \in \N$, the type $X_{t}(n)$ of the particle at level $n$ do not evolve between reproduction events which affect level $n$.
\end{itemize}
\begin{rem}
\label{definition}
 The integrability condition $\int_{(0,1]} x^2 \nu(dx)<\infty$ ensures that finitely many reproduction events change the type of a particle at a given level in a finite time interval, ensuring the preceeding definition makes sense.
\end{rem}

For a fixed $t\geq 0$, the sequence $(X_t(n), n \in \N)$ is exchangeable according to Proposition 3.1 of \cite{DO99}. This allows to define the random probability measure $R_t$ on $E$ as the de Finetti measure of the sequence $(X_t(n), n \in \N)$:
\begin{equation}
\label{definetti}
R_t(dx) = \lim_{N \to \infty}   \frac{1}{N} \ \sum_{n=1}^{N} \delta_{X_t(n)} (dx)
\end{equation}
and this defines a random process indexed by the set of non negative real numbers $\R^+$ since a c\`adl\`ag version of the process (where the space $\M_{f}(E)$ of finite measures on $E$ is endowed with the topology of weak convergence) is shown to exist in \cite{DO99}.
The process $R$ is called the Generalized Fleming Viot process (GFV process for short) without mutation. 
We stress that, conditionally given $R_t$, the random variables $(X_t(n), n \in \N)$ on $E$ are independent and identically distributed according to the probability measure $R_t$ thanks to de Finetti's Theorem. \\
We will denote by $\P$ the law of $X$. We now introduce the relevant filtrations we will work with:
\begin{itemize}
 \item $\left( \F_t = \sigma \{(X_{s}(n),n \in \N) , 0 \leq s \leq t \} \right)$ corresponds to the filtration of the particle system. 
 \item $\left( \G_t= \sigma \{ R_s, 0 \leq s \leq t \} \right)$ corresponds to the filtration of the measure-valued process $R$.
\end{itemize}
Notice that $X$ is a Markov process with respect to the filtration $\F$, and  that $R$ is a Markov process with respect to the filtration $\G$.

\subsection{A pathwise construction of an $h$-transform}
\label{section1}

\subsubsection{Results}
The proofs of the results contained in this Subsection may be found in \ref{proofs}.
Fix $1 \leq K \leq K'$. We assume from now on and until the end of Section \ref{section2} that:  
\begin{equation}
\label{notnull} \E(\prod_{i=1}^{K} R_0\{i\})>0, 
\end{equation}
to avoid empty definitions in the following. Recall the definition of the particle system $X$ associated with $R$ in Section \ref{spatial}. We define from $X$ a new particle system $X^h$ as follows:
\begin{itemize}
 \item[(i)] The finite sequence $\left( X_0^h(j), 1 \leq j \leq K \right)$ is a uniform permutation of $\{1, \ldots, K\}$, and, independently, the sequence $\left(X_{0}^h(j), j \geq K+1 \right)$ is exchangeable with asymptotic frequencies $R_{0}^{H}$, where $R^H_0$ is the random probability measure with law:
\begin{equation*}
 \P(R^H_0 \in A)= \E\left(\ind_{A}(R_0) \frac{\prod_{i=1}^{K} R_0\{i\}}{\E(\prod_{i=1}^{K} R_0\{i\})}\right). \end{equation*}
\item[(ii)] The reproduction events are given by the \textit{restriction} of the Poisson point measure $N$ (defined as in Subsection \ref{spatial}) to $V  \equal \left\{ (s,\pi), \pi_{| [K]}= \left\{ \{1\}, \{2\}, \ldots,\{ K \} \right\} \right\} $, where $\pi_{| [K]}$ is the restriction of the partition $\pi$ of $\N$ to $\{1,\ldots,K\}$, that is the atoms of $N$ for which the reproductions events do not involve more than one of the first $K$ levels.
\end{itemize}

Remark \ref{definition} ensures that this definition of the particle system $X^h$ makes sense.

\begin{rem}
Note that the particle system $\left( X_0^h(j), j \geq 1\right)$ is no more exchangeable due to the constraint on the $K$ initial levels. Nevertheless, the  particle system $\left( X_0^h(j), j >K \right)$ is still exchangeable, and we shall view the $K$ first levels as $K$ independent sources of immigration in forthcoming Section \ref{immigration}.
\end{rem}

We also need the definition of the lowest level $L(t)$ at which the first $K$ types appear:
\begin{equation}
\label{defL}
L(t)=\inf \lbrace i \geq K, \{1, \ldots, K\} \subset \{ X_t(1), \ldots, X_t(i)  \} \rbrace, 
\end{equation}
with the convention that $\inf \{ \varnothing \}  = \infty$. The random variable $L(0)$ is finite if and only if $\prod_{i=1}^{K} R_0\{i\}>0$, $\P$-a.s., thanks to de Finetti's Theorem. The process $L(t)$ is $\F_t$ measurable, but not $\G_t$ measurable. Notice the random variable $L(t)$ is an instance of the famous coupon collector problem, based here on a random probability measure.
We define, for $i \geq 1$, the \textit{pushing rates} $r_i$ at level $i$:
\begin{align*}
r_{i}&=\frac{i(i-1)}{2} \; c + \int_{(0,1]} \nu(dx) \left(1-(1-x)^i- ix(1-x)^{i-1}\right).
\end{align*} 
Notice that $r_1 =0$  and that $r_i$ is finite for every $i \geq 1$ since $\int_{(0,1]} x^2 \ \nu(dx)<\infty$. From the construction of the lookdown particle system, these pushing rates $r_i$ may be understood as the rate at which a type at level $i$ is pushed up to higher levels (not necessarily $i+1$) by reproduction events at lower levels.  
Let us define a process $Q= (Q_t, t \geq 0)$ as follows:
$$Q_t =\frac{\ind_{\{L(t)=K\}} }{\P(L(0)=K)} \expp{r_{K} t}.$$
\begin{lem}
\label{kimura_genealogy}
The process $Q= (Q_t, t \geq 0)$ is a non negative $\F$-martingale, and 
\begin{equation}
\forall A \in \F_t, \ \P(X^h \in A)= \E \left( \ind_{A}(X) \ Q_t \right).
\end{equation}
\end{lem}

By projection on the smaller filtration $\G_t$, we deduce Lemma \ref{martingale_mult}. We need the following definition of the process:
 $$M_t = \frac{\prod_{i=1}^{K} R_t\{i\}}{\E(\prod_{i=1}^{K} R_0\{i\})} \expp{r_{K} t}.$$
\begin{lem}
\label{martingale_mult}
The process $M=(M_t, t \geq 0)$ is a non negative $\G$-martingale.
\end{lem}
This fact allows to define the process $R^H= (R^H_t, t \geq 0)$ absolutely continuous with respect to $R= (R_t, t \geq 0)$ on each $\G_t$, $t \geq 0$, with Radon Nykodim derivative:
\begin{equation}
\label{kimuraeq}
\forall A \in \G_t, \ \P(R^H \in A)= \E\left(\ind_{A}(R) \ M_t \right).
\end{equation}
The process $R^H$ is the product type $h$-transform of interest.
\begin{rem}
Definition \reff{kimuraeq} agrees with the definition of $R_0^H$.
\end{rem}
\begin{rem}
Intuitively, the ponderation by $M$ favours the paths in which the $K$ first types are present in equal proportion.
\end{rem}
We shall deduce from Lemma \ref{kimura_genealogy} and Lemma \ref{martingale_mult} the following Theorem.
\begin{theo}
\label{kimura_restriction}
Let $1 \leq K \leq K'$. We have that:
\begin{itemize}
 \item[(a)] The limit of the empirical measure:
\begin{equation*}
R^h_t(dx) \equal \lim_{N\to \infty} \frac{1}{N} \sum_{n=1}^{N} \delta_{X_{t}^h(n)}(dx)
\end{equation*}
exists a.s.
\item[(b)] The process $(R^h_t, t \geq 0)$ is distributed as $(R^H_t, t \geq 0)$.
\end{itemize}
\end{theo}
Let us comment on these results.
The process $X^h$ is constructed by changing the initial condition and forgetting (as soon as $K\geq 2$) specific reproduction events in the lookdown particle system of $X$. Lemma \ref{kimura_genealogy} tells us that this procedure selects the configurations of $X$ in which the $K$ lowest levels are filled with the $K$ first types at initial time without any ``interaction'' between these $K$ first levels at a further time. Theorem \ref{kimura_restriction} tells us that the process $R^h$ constructed in this way is an $h$-transform of $R$ and Lemma \ref{martingale_mult} yields the following simple probabilistic interpretation of the Radon Nikodym derivative in equation \reff{kimuraeq}: the numerator is proportional to the probability that the $K$ lowest levels are occupied by the $K$ first types at time $t$, whereas the denominator is proportional to the probability that the $K$ lowest levels are occupied by the $K$ first types at time $0$. We shall see in Section \ref{conditioned} that the processes $X^h$ and $R^h$ also arise by conditioning the processes $X$ and $R$ on coexistence of the first $K$ types.

\subsubsection{Proofs}
\label{proofs}

\begin{proof}[Proof of Lemma \ref{kimura_genealogy}]
From the de Finetti theorem, conditionally on $R_t$, the random variables $(X_{t}(i), i \in \N)$ are independent and identically distributed according to $R_t$. This implies that: 
\begin{equation}
\label{key}
\P(L(t)=K |\G_t )=K! \; \prod_{i=1}^{K} R_t\{i\}. 
\end{equation}
In particular, we have:
\begin{equation*}
\P(L(0)=K)=K! \; \E(\prod_{i=1}^{K} R_0\{i\}), 
\end{equation*}
which, together with \reff{notnull}, ensures that $Q_t$ is well defined.

Then, let us define $W=\lbrace \pi, \pi_{| [K]} = \{ \{1\}, \{2\}, \ldots,\{ K \} \} \rbrace$, and $V_t = \lbrace (s,\pi), 0 \leq s \leq t, \pi \in W \rbrace $, and also the set difference $W^c= \p_{\infty} \setminus W$ and $V_t^c = \lbrace (s,\pi), 0 \leq s \leq t, \pi \in W^c \rbrace$. 
We observe that:
\begin{itemize}
\item from the de Finetti Theorem, the law of $X_0^h$, as defined in (i), is that of $X_0$ conditioned on $\{L(0)=K\}$. 
\item The law of the restriction of a Poisson point measure on a given subset is that of a Poisson point measure conditioned on having no atoms outside this subset: thus $N$ conditioned on having no atoms in $V_t^c$ (this event has positive probability) is the restriction of $N$ to $V_t$.
\end{itemize}
Since the two conditionings are independent, we have, for $A \in \F_t$:
\begin{align}
\P(X^h \in A) &= \P ( X \in A | \{L(0)=K\} \cap  \{N(V_t^c)=0\} ) \nonumber\\
&= \E \bigg(\ind_{A}(X) \frac{\ind_{\{L(0)=K\} \cap \{N(V_t^c)=0\} }}{\P(L(0)=K) \P(N(V_t^c)=0)} \bigg)  \label{3} 
\end{align}
We compute: 
\begin{align*}
\mu(W^c) &=\mu^k(W^c)+\int_{(0,1]} \nu(dx) \rho_{x}(W^c) \\
&=\frac{K(K-1)}{2} \; c \; + \; \int_{(0,1]} \nu(dx) \left(1-(1-x)^K- Kx(1-x)^{K-1}\right)\\
&=r_K. 
\end{align*}
This implies from the construction of $N$ that:
\begin{equation}
\label{3'}
\P(N(V_t^c)=0)= \expp{-\mu(W^c) t}=e^{-r_K t}. 
\end{equation}
Notice that 
\begin{equation}
\label{3''}
\{L(t)=K\} =\{L(0)=K\} \cap  \{N(V_t^c)=0\}.
\end{equation}
From \reff{3}, \reff{3'} and \reff{3''}, we deduce that:
\begin{align*}
\P(X^h \in A) &=
 \E \bigg(\ind_{A}(X) \frac{\ind_{\{L(t)=K \}}}{\P(L(0)=K)} \expp{r_K t}\bigg) =  \E \bigg(\ind_{A}(X) Q_t\bigg) .
\end{align*}
\end{proof}

\begin{proof}[Proof of Lemma \ref{martingale_mult}]
We know from Lemma \ref{kimura_genealogy} that $(Q_t, t \geq 0)$ is a $\F$-martingale. Since $\G_t \subset \F_t$ for every $t\geq 0$,
we deduce that $(\E(Q_t|\G_t), t \geq 0)$ is a $\G$-martingale. 
Now, we notice that:
\begin{align*}
\E(Q_t|\G_t) &= \E\left(\frac{\ind_{\{L(t)=K \}}}{\P(L(0)=K)} \expp{r_K t}| \G_t \right) = \frac{\prod_{i=1}^{K}  R_t\{i\}}{\E(\prod_{i=1}^{K} R_0\{i\})} \expp{r_K t} =M_t.
\end{align*}
using \reff{key} for the second equality.

\end{proof}

\begin{proof}[Proof of Theorem \ref{kimura_restriction}]
From Lemma \ref{kimura_genealogy}, $X^h$ is absolutely continuous with respect to $X$ on $\F_t$. The existence of the almost sure limit of the empirical measure claimed in point (a) follows from \reff{definetti}.
We now project on $\G_t$ the absolute continuity relationship on $\F_t$ given in Lemma $\ref{martingale_mult}$. Let $A \in \G_t$:
\begin{align*}
\P(R^h \in A)= \E \left(\ind_{A}(R) Q_t \right)  &= \E \left(\ind_{A}(R) \E(Q_t|\G_t) \right) = \E \left(\ind_{A}(R) M_t \right) =  \P(R^H \in A), 
\end{align*}
where we use Lemma \ref{kimura_genealogy} for the first equality and the definition of $R^H$ for the last equality. This proves point (b).
\end{proof}

\subsection{The $h$-transform as a conditioned process}
\label{conditioned}
We study the conditionings associated  with the $h$-transforms.

Let  $1 \leq K \leq K'$. Assumption \reff{notnull} allows us to define a family of processes $R^{(\geq t)}$ on $\F$ by:
\begin{equation*}
\forall A \in \G_t, \  \P(R^{(\geq t)} \in A) = \P \big(R \in A | \prod_{i=1}^{K} R_{t}\{i\} \neq 0 \big),
\end{equation*}
and the associated particle system $X^{(\geq t)}$ on $\G$ by:
\begin{equation*}
\forall A \in \F_t, \  \P(X^{(\geq t)} \in A)=  \P \big(X \in A | \prod_{i=1}^{K} R_{t}\{i\} \neq 0 \big).
\end{equation*}
The process $R^{(\geq t)}$ thus corresponds to the process $R$ conditioned on coexistence of each of the \textit{first} $K$ types at time $t$. 
It is not easy to derive the probabilistic structure of the particle system $X^{(\geq t)}$ on all $\F_t$. Nevertheless, for fixed $s\geq 0$, the probabilistic structure of $X^{(\geq t)}$ on the sigma algebra $\F_s$ simplifies as $t$ goes to infinity.

The following Theorem may be seen as a generalization of Theorem 3.7.1.1 of Lambert \cite{L08} which, building on the work of Kimura \cite{KI57}, deals with the case $\nu=0$. We write $\P_i$ for the law of $L$ (defined in \reff{defL}) conditionally on  $\{L(0)=i\}$. For $I$ an interval of $\R^+$ and $F$ a process indexed by $\R^+$, we denote by $F_{I}$ the restriction of $F$ on the interval $I$.
\begin{theo}
\label{kimura_conditioned}
Let $s \geq 0$ be fixed. Assume that 
\begin{equation}
\label{condition} \lim_{t \to \infty} \frac{\P_{K+1}(L(t)<\infty)}{\P_{K}(L(t)<\infty)}=0. 
\end{equation}
We have:
\begin{itemize}
\item[(i)] The family of processes $X_{[0,s]}^{(\geq t)}$ weakly converges as $t \to \infty$ towards the process $X^h_{[0,s]}$.
\item[(ii)] The family of processes $R_{[0,s]}^{(\geq t)}$ weakly converges as $t \to \infty$ towards the process $R^h_{[0,s]}$.
\end{itemize}
\end{theo}

\noindent Lemma \ref{c} gives a sufficient condition for \reff{condition} to be satisfied.

\begin{rem}
 The case $K=1$ corresponds to a non degenerate conditioning since the event $\{R_{t}\{1\} \neq 0 \mbox{ for every } t\}$ has positive probability under \reff{notnull}.
\end{rem}

\begin{rem}
\label{CDIrem}
Assume $K \geq 2$. The following property
$$\mbox{(CDI)  \ \ \ \ } \P(\inf{\{t>0, L(t)=\infty \}} < \infty)=1,$$
is independent of $K$ (used to define $L$ in \reff{defL}).
(CDI) property is easily seen to correspond to the Coming Down from Infinity property for the $\Lambda$-coalescent associated with the GFV process $R$ (whence the acronym (CDI)). We refer to Schweinsberg \cite{SC01} for more details about this property.
We conjecture (CDI) property is equivalent to our assumption \reff{condition}. We were unable to prove it. 
\end{rem}

\begin{rem}
It should still be possible to interpret the processes $X^h$ and $R^h$ as conditioned processes, without assuming \reff{condition}. In that more general case, $X^h$ should correspond to $X$ conditioned by the event $\{ \liminf_{t \to \infty} \prod_{i=1}^{K} R_{t}\{i\} >0 \}$ (which has null probability as soon as $K \geq 2$).
\end{rem}

\begin{proof}
First observation is that, from the Kingman's paintbox construction for exchangeable random partition, we have: $\prod_{i=1}^{K} R_t\{i\} \neq 0$  if and only if  $L(t)<\infty$, $\P$ a.s.
This gives, for any $A \in \F_s$:
\begin{align*}
\P \bigg( A | \prod_{i=1}^{K} R_t\{i\} \neq 0 \bigg)
&= \frac{\P \left(A \cap \{\prod_{i=1}^{K} R_t\{i\} \neq 0\} \right)} {\P(\prod_{i=1}^{K} R_t\{i\} \neq 0)} \\
&= \frac{\P \left(A \cap \{L(t)<\infty\} \right)} {\P(L(t)<\infty)}. \\
\end{align*}
Now, using the Markov property, we have:
\begin{align*}
 \P(A \cap \{ L(t) < \infty\} ) 
&= \P(A  \cap \{ L(s)=K \} \cap \{ L(t) < \infty\} ) +  \P(A \cap \{ L(s)\geq K+1 \} \cap \{ L(t) < \infty\} ) \\
&= \E (\ind_{A \cap \{ L(s)=K \}} \;  \P_{K}(\tilde{L}(t-s) < \infty )) +  \E (\ind_{A \cap \{ L(s) \geq K+1 \}} \; \P_{L(s)}(\tilde{L}(t-s) < \infty )). \\ 
\end{align*}
where $\tilde{L}$ is an independent copy of $L$. 

Let $\ell \in \N$. We can couple the processes $L$ under $\P_{\ell}$ and $L$ under $\P_{\ell+1}$ on the same lookdown graph  by using the same reproduction events. 
More precisely, imagine that we distinguish the particles at level $\ell$ and $\ell+1$ at initial time, giving each of them a special type shared by no other particles. Then the first levels at which these two types may be found at time $t$ yields a coupling of $L(t)$ under $\P_{\ell}$ and $L(t)$ under $\P_{\ell+1}$.
Let us denote by $(L_\ell,L_{\ell+1})$ this coupling: $L_\ell$ is distributed as $L$ under $\P_\ell$ and $L_{\ell+1}$ is distributed as $L$ under $\P_{\ell+1}$. By the ordering by persistence property of the lookdown graph, we have that, for every $t\geq0$: $$L_\ell(t) \leq L_{\ell+1}(t),$$ whence:
\begin{equation}
\label{persistence} \P_{\ell+1}(L(t)<\infty) \leq \P_{\ell}(L(t)<\infty)
\end{equation}
for every integer $\ell$.
Therefore, we have:
\begin{align*}
\E (\ind_{A \cap \{ L(s) \geq K+1 \}} \P_{L(s)}(\tilde{L}(t-s) < \infty )) & \leq  \P (A \cap \{ L(s) \geq K+1 \}) \;  \P_{K+1}(L(t-s) < \infty ).
\end{align*}
Our assumption \reff{condition} now implies:
\begin{align*}
\frac{\P(A \cap   \{ L(t)<\infty \} )}{\P_{K}(L(t-s) < \infty)}
&\underset{t \to \infty}{\rightarrow}  \P (A \cap \{ L(s)=K \}).
\end{align*}
Setting $A=\Omega$, this also yields:
\begin{align*}
\frac{\P( L(t)<\infty )}{\P_{K}(L(t-s) < \infty)}
&\underset{t \to \infty}{\rightarrow}  \P ( L(s)=K).
\end{align*}
Taking the ratio, we find that:
\begin{align*}
\frac{\P \left(A \cap \{L(t)<\infty\} \right)} {\P(L(t)<\infty)} 
 \underset{t \to \infty}{\rightarrow}  \frac{\P (A \cap \{ L(s)=K \})}{\P(L(s)=K)}.
\end{align*}
We also have that $\P(L(s)=K)=\P(L(0)=K) \expp{-r_K s}$ since $Q$ is a $\G$-martingale from Lemma \ref{kimura_genealogy}. Altogether, we find that:
\begin{align*}
\lim_{t \to \infty} \P(A  |  \prod_{i=1}^{K} R_t\{i\} \neq 0 \})= \E\left(\ind_{A}(X) \frac{\ind_{\{L(s)=K\}}}{\P(L(0)=K)} \expp{r_K s}\right)=\P(X^h\in A)
\end{align*}
where the last equality corresponds to Lemma \ref{kimura_genealogy}. This implies the convergence in law of $X^{(\geq t)}$ towards $X^h$ as $t \to \infty$. This proves the first point.\\
The proof of (ii) is similar to the one for (i).
\end{proof}

\begin{rem}
Having introduced in the preceeding proof the coupling $(L_{K},L_{K+1})$, we may complete the Remark \ref{CDIrem}: 
It is possible to prove that, if (CDI) holds and for each $t \geq 0$,  $$ \left( j \to \P(L_{K+1}(t)< \infty|L_{K}(t) \leq j) \right) \mbox{ is non increasing}$$ 
then \reff{condition} holds. 
\end{rem}

\noindent We now give a sufficient condition for \reff{condition} to be satisfied.

\begin{lem}
\label{c}
If $\sum_{j \geq K} \frac{1}{r_j}<\infty$, then \reff{condition} holds.
\end{lem}

\begin{proof}
A lower bound for $\P_K(L(t) < \infty)$ is easily found:
\begin{equation}
\label{lowerbound}
\expp{-r_K t} = \P_K(L(t) =K) \leq \P_K(L(t) < \infty). 
\end{equation}
We now look for an upper bound for $\P_{K+1}(L(t) < \infty)$. Recall the non decreasing pure jump process $L$ jumps with intensity $r_j$ when $L=j$. %

We may write, under $\P_{K+1}$:
$$\sup{\{t, L(t)<\infty\}}=\sum_{j \geq {K+1}} \tilde{T}_j$$
where, conditionally given the range of the function $L$ we shall denote by $L^{K+1}<L^{K+2}<\ldots$, the sequence $(\tilde{T}_j, j \geq K+1)$ is a sequence of independent exponential random variable with parameter $r_{L^j}$. Since $(r_j)_{j \geq K+1}$ forms an increasing sequence and the function $L$ has jumps greater or equal to one, we have for each $j \geq K+1$, 
\begin{equation}
\label{ineq}
r_{L^j} \geq r_j. 
\end{equation}
Let $(T_j, j \geq K+1)$ be a sequence of independent exponential random variables with parameter $(r_j, j \geq K)$.
We compute, for $0  < \lambda < r_{K+1}$ :
\begin{align*}
\P_{K+1}(L(t) < \infty) & = \P(\sum_{j \geq {K+1}} \tilde{T}_j > t) \\
&\leq  \P(\sum_{j \geq {K+1}} T_j > t) \\
&= \P(\exp{ (\lambda \sum_{j \geq {K+1}} T_j)} > \exp{ \left(\lambda t\right) } ) \\
& \leq \exp{ \left(-\lambda t \right)} \; \E \big(  \exp{\big( \lambda \sum_{j \geq K+1} T_j \big) }  \big) \\
& =\exp{ \left(-\lambda t\right) } \prod_{j \geq K+1} \frac{r_j}{r_j-\lambda} \\
& =  \exp{ \big(-\lambda t +\sum_{j \geq K+1} \log\big(1+\frac{\lambda}{r_j-\lambda}\big) \big) } \\
& \leq   \exp{\big( -\lambda t + \lambda  \sum_{j \geq K+1}  \frac{1}{r_j-\lambda} \big) }, \\
\end{align*}
where we use \reff{ineq} for the first inequality and the Markov inequality for the second inequality.
From the assumption, $\sum_{j \geq K+1} 1/r_j$ is finite, which implies also  that $\sum_{j \geq K+1}  1/(r_j-\lambda)$ is finite. Taking $\lambda= (r_K+r_{K+1})/2$, we obtain that:
\begin{align}
\label{upperbound}
\P_{K+1}(L(t) < \infty) < C \exp{ \left( - \frac{r_K+r_{K+1}}{2} t \right)}
\end{align}
for the finite constant $C=\exp{ \lambda   \sum_{j \geq K}  1/(r_j-\lambda)}$ associated with this choice of $\lambda$.
Altogether, using \reff{lowerbound} and \reff{upperbound}, we have that:
\begin{align*}
0\leq  \frac{\P_{K+1}(L(t)<\infty)}{\P_{K}(L(t)<\infty)} \leq   C \exp{\left(- \frac{r_{K+1}-r_{K}}{2} t \right)}.
\end{align*}
Letting $t$ tend to $\infty$, we get the required limit.
\end{proof}
As an immediate corollary, we get the following result, which ensures that \reff{condition} is satisfied in the most interesting cases. 
\begin{cor}
If $(c\neq 0)$, or $\big(c=0$ and $\nu(dx) =  x^{-1-\alpha} (1-x)^{\alpha -1} \ind_{(0,1)}(x) dx$ for some $1 <\alpha < 2 \big)$, then \reff{condition} holds.
\end{cor}

\begin{proof}
If $(c \neq 0),$ $r_j \geq c  j(j-1)/2$, and thus $\sum_{j \geq K} 1/r_j < \infty$.
If $\big(c=0$ and \linebreak $\nu(dx) = x^{-1-\alpha} (1-x)^{\alpha -1} \ind_{(0,1)}(x) dx$ for some $1 <\alpha < 2\big)$,
using the equality:
$$r_{j+1}-r_{j}= \int_{(0,1]} j (1-x)^{j-1} x^{2} \nu(dx);$$
one may find at Lemma 2 of Limic and Sturm \cite{Limic},
we get $$r_{j+1}-r_{j}=j \mbox{ Beta}( 2-\alpha,j+\alpha-1) \underset{j \to \infty}{\sim}  \Gamma(2-\alpha) j^{\alpha-1},$$
where, as usual, $u_j\underset{j \to \infty}{\sim} v_j$ means that $v_j \neq 0$ for $j$ large enough and $\lim_{j \to \infty} u_j/v_j=1$. We conclude that $r_j \underset{j \to \infty}{\sim} \Gamma(2-\alpha)j^{\alpha} /\alpha$, and then:
$\sum_{j \geq K} 1/r_j < \infty$, since $1< \alpha <2$.

Lemma \ref{c} allows to conclude that \reff{condition} holds in both cases.

\end{proof}

\subsection{The immigration interpretation}
\label{immigration}

We develop further the two following examples: 
\begin{enumerate}
\item[(i)] $K=K'=2$: this amounts (provided condition \reff{condition} is satisfied) on conditioning a two-type GFV process on coexistence of each type.
\item[(ii)] $1=K < K'=2$: this amounts (provided \reff{condition} is satisfied) on conditioning a two-type GFV process on absorbtion by the first type.  
\end{enumerate}
We regard the $K (=1 \mbox{ or } 2)$ first level particles in $X^h$ as $K$ external sources of immigration in the population now assimilated to the particle system $(X^h(n), n \geq K+1)$ and
decompose the generator of the process $R^h$ accordingly.
We refer to Foucart \cite{F10} for a study of GFV processes with one source of immigration ($K=1$ here).

Since $K'=2$, the resulting probability measure-valued process $R=(R_t, t\geq 0)$ and $R^h= (R^h_t, t \geq 0)$ on $\{1,2\}$ may be simply described by the $[0,1]$-valued processes $R\{1\}=(R_t\{1\}, t \geq 0)$ and $R^{h}\{1\}=(R^h_t\{1\}, t \geq 0)$ respectively. 
For the sake of simplicity, we will just write $R$ for $R\{1\}$ and $R^h$ for $R^h\{1\}$ respectively.

We recall that the infinitesimal generator of $R$ is given by:
\begin{equation*}
Gf(x)=\frac{1}{2} c x(1-x) f''(x)  + x \int_{(0,1]}  \nu(dy) [ f(x(1-y)+y)-f(x) ] + (1-x) \int_{(0,1]}  \nu(dy) [ f(x(1-y))-f(x) ]
\end{equation*}
for all $f \in \C^{2}([0,1])$, the space of twice differentiable functions with continuous derivatives, and $x \in [0,1]$, see Bertoin and Le Gall \cite{BLG}. 
\subsubsection{ We assume $K=K'=2$} 
\label{K=K'}
We define, for $f \in \C^{2}([0,1])$, and $x \in [0,1]$:
\begin{align*}
G^{0}f(x) = c (1-2x)f'(x) + \int_{(0,1]} y(1-y) \nu(dy) [f(x(1-y)+y)-f(x)] &\\
   +  \int_{(0,1]} y(1-y) \nu(dy) [f(x(1-y))-f(x)]& , 
\end{align*}
and
\begin{align*}
G^{1}f(x)  = \frac{1}{2} c x(1-x)f''(x)+ x \int_{(0,1]}  (1-y)^2 \nu(dy) [f(x(1-y)+y))-f(x)] & \\
 + (1-x) \int_{(0,1]} (1-y)^2 \nu(dy) [f(x(1-y))-f(x)]&.
\end{align*}
\begin{proposition}
\label{diffusion}
Assume $K=K'=2$.
The operator $G^{0}+G^{1}$ is a generator for $R^h$. 
\end{proposition}


\begin{rem}
When the measure $\nu$ is null, the process $R$ is called a Wright Fisher (WF in the following) diffusion. In that case, the process $R^h$ may be seen as a WF diffusion with immigration, where the two first level particles induce continuous immigration (according to $G^0$) of both types 1 and 2 in the original population (which evolves according to $G^1=G$ in that case).

When the measure $\nu$ is not null, the process $R^h$ may still be recognized as a GFV process with immigration, but the generator $G^1$ is no more that of the initial GFV process $G$: the two first level particles induce  both continuous and discontinuous immigration   (according to $G^0$) of types 1 and 2 in a population with a reduced reproduction (the measure $\nu(dy)$ is ponderated by a factor $(1-y)^2 \leq 1$ in $G^1$).
\end{rem}

\begin{proof}
Let us denote by $G^h$ the generator of $R^h$. The process $R^{h}$ is the Doob $h$-transform of $R$ for the following function $H$, which is space time harmonic according to Lemma \ref{martingale_mult}:
\begin{equation*}
H(t,x)= x(1-x)\expp{r_2 t}. 
\end{equation*}
From the definition of the generator, for $f \in \C^{2}([0,1])$, and $x \in [0,1]$:
\begin{equation*}
f(R_t) H(t,R_t)- f(R_0) H(0,R_0) - \int_{0}^{t} ds \; G(H(s,.) f)(R_s) - \int_{0}^{t} ds \; \partial_{t} H(.,R_s)(s) f(R_s)
\end{equation*}
is $\G$ martingale, where in the first integrand $G$ acts on $x \to f(x) H(s,x)$. Therefore, on $\{ H(0,R_0)\neq0\}$, the process
\begin{equation*}
\frac{H(t,R_t)}{H(0,R_0)} f(R_t) - f(R_0) - \int_{0}^{t} ds \; \frac{H(s,R_s)}{H(0,R_0)} \frac{G( H(s,.) f)(R_s)}{H(s,R_s)} - \int_{0}^{t} ds \; \frac{H(s,R_s)}{H(0,R_0)}  \frac{\partial_{t} H(s,R_s)}{H(s,R_s)} f(R_s)
\end{equation*}
is again $\G$ martingale under $\P$. This implies that:
\begin{equation*}
f(R^h_t) - f(R^h_0) - \int_{0}^{t} ds \;  \frac{G(H(s,.) f )(R^h_s)}{H(s,R^h_s)} - \int_{0}^{t} ds \;  \frac{\partial_{t} H(s,R^h_s)}{H(s,R^h_s)} f(R^h_s)
\end{equation*}
is a $\G$ martingale under $\P$. We thus have:
\begin{align}
\label{hgenerator}  
G^{h}f(x)&=\bigg(\frac{G(H(t,.) f)+ \partial_t H(t,.) f}{H(t,.)}\bigg)(x) \\ 
&= \frac{G(H(t,.) f)}{H(t,.)}(x) + r_{2} f(x), \nonumber
\end{align}
a formula which holds for any $t \geq 0$. A simple computation completes the proof of the proposition. 
\end{proof}


Using the particle system $X^h$, we also have the following intuitive interpretation of the generator $G^h$ in the case of a pure jump GFV process $(c=0)$. 
Let us decompose the intensity measure $\nu$ as follows:
$$\nu(dy)= 2 y(1-y)\nu(dy)+ (1-y)^2\nu(dy) + y^2 \nu(dy).$$

\begin{enumerate}
\item The first term is the sum of the two measures $y(1-y)\nu(dy)$ appearing in each integrand of the generator $G^{0}$ and each of these measures corresponds to the intensity of the reproduction events involving level $1$ and not level $2$, or level $2$ and not level $1$ (these events have probability $y(1-y)$ when the reproduction involves a fraction $y$ of the population).
\item  The second term is the measure $(1-y)^2\nu(dy)$  appearing in the generator $G^{1}$ and corresponds to the intensity of the reproduction events involving neither level $1$ nor level $2$ (this event has probability $(1-y)^2$ when the reproduction involves a fraction $y$ of the population).
\item The third term does not appear in the generators $G^0$ and $G^1$: it corresponds to the intensity of the reproduction events involving both level $1$ and $2$, and these events have been discarded in the construction of $X^h$.
\end{enumerate}

\subsubsection{We assume $K=1,K'=2$} 
\label{K'=1}
Note that the case $K=1$ differs from the case $K=2$, since the event $\{R_{t} \neq 0 \mbox{ for every } t\}$ has positive probability under \reff{notnull}. Let us  define, for $f \in \C^{2}([0,1])$, and $x \in [0,1]$:
\begin{align*}
I^{0}f(x)& = c (1-x)f'(x)+ \int_{(0,1]} y \nu(dy) [f(x(1-y)+y)-f(x)] 
\end{align*}
and
\begin{align*}
I^{1}f(x) = \frac{1}{2} c x(1-x)f''(x)+  x  \int_{(0,1]} (1-y) \nu(dy) [f(x(1-y)+y))-f(x)] & \\
 + (1-x) \int_{(0,1]} (1-y) \nu(dy) [f(x(1-y))-f(x)]&.
\end{align*}
We can then prove the analog of Proposition \ref{diffusion} in that setting.
\begin{proposition}
\label{diffusion2}
Assume $K=1, K'=2$.
The operator $I^{0}+I^{1}$ is a generator for the Markov process $R^h$.
\end{proposition}
In particular, we recover the well known fact that a WF diffusion conditioned on fixation at 1 (that is, $R_t= 1$ for $t$ large enough) may be viewed as a WF process with immigration, see \cite{D08} for instance.

\begin{proof}
The proof is similar to that of Proposition \ref{diffusion}. Here we use an $h$-transform with the function
\begin{equation*}
H(t,x)= x. 
\end{equation*}
This function is space time harmonic according to Lemma \ref{martingale_mult} (recall $r_1=0$).
\end{proof}
Here again, we may have obtained the generator of $R^h$ without computation. We explain the procedure in the case of a pure jump GFV process $(c=0)$. 
We decompose $\nu$ as follows:
$$\nu(dy)= y\nu(dy)+ (1-y)\nu(dy).$$
\begin{enumerate}
 \item The first term is the measure $y\nu(dy)$ appearing in the generator $I^{0}$. This is the intensity of the reproduction events involving level $1$ particle. We interpret them as immigration events.
\item  The second term is the measure $(1-y)\nu(dy)$  appearing in the generator $I^{1}$. This is the intensity of the reproduction events not involving level $1$ particle. We interpret them as reproduction events.
\item Summing the two measures $y\nu(dy)$ and $(1-y)\nu(dy)$, we recover this time the full measure $\nu(dy)$ since no reproduction events are discarded in the case $K=1$.
\end{enumerate}

\subsubsection{Intertwining} 
We assume $K'=2$ and $\nu=0$ (for the sake of simplicity). In this Subsection, we observe that the processes $L$ and $R$ are intertwined in the sense of Rogers and Pitman \cite{PR81}.  This adds another decomposition to the striking one of Swart, see \cite{SW11}, which does not admit a clear interpretation from the lookdown particle system.

\noindent
In fact, we find it more convenient to prove rather that $L^1$ and $R$ are intertwined, where:
$$L^{1}(t)=\inf{ \{ i \geq 1, 1 \in \{X_{t}(1), \ldots, X_{t}(i) \} \}}$$
is the first level occupied by a type 1 particle. The process $L^1$ is valued in $\N \cup \{\infty\}$, and jumps by $1$ at rate $c \frac{\ell(\ell-1)}{2}$ when at $\ell$. In particular, $1$ is an absorbing point for $L^1$. Notice also that the process $R^h$ studied in Subsection \ref{K'=1} is the process $R$ conditioned on $\{L^1=1\}$.

Let us define the following kernel:
$$\hat{K}(x,\ell)= (1-x)^{\ell-1} x \ind_{(0,1]}(x) + \ind_{\{\infty\}}(\ell) \ind_{\{0 \}}(x), \quad x \in [0,1], \ell \in \N \cup \{\infty\}.$$
acting on function $f(x,\ell)$ as follows:
$$\hat{K}f(x)=\sum_{1 \leq \ell <\infty} K(x,\ell) f(x,\ell).$$
We slightly abuse of notation by still denoting by $G$ the generator of the Wright Fisher diffusion:
$$ Gf(x)= \frac{1}{2} c x (1-x) f''(x)$$
acting on $f \in \C^{2}([0,1])$. We denote $\hat{G}$ the generator defined for $\ell<\infty$ by:
$$ \hat{G}f(x,\ell)= \frac{1}{2} c x (1-x) \partial_{xx} f(x,\ell) + c \left[ (1-x)-(\ell-1)x \right] \partial_{x} f(x,\ell) +   c \frac{\ell(\ell-1)}{2}    \left[ f(x,\ell+1) -  f(x,\ell) \right] 
$$
and for $\ell=\infty$ by:
$ \hat{G}f(x,\infty)=0$. This generator acts on functions $f$ such that $f$, as a function of $x$, belongs to $\C^{2}([0,1])$. 
The intertwining relationship reads as follows.
\begin{proposition}
Let $f$ be in the domain of  $\hat{G}$ and $x \in [0,1].$ The kernel $\hat{K}$ intertwins the generators $G$ and $\hat{G}$ in the sense that: $$\hat{K} \hat{G}(f)(x)= G \hat{K}(f)(x).$$
\end{proposition}
The proof consists in a long but simple calculation and is eluded. A similar intertwining relation also holds for $\nu \neq 0$, but the generator $\hat{G}$ is then more complicated (because $L^1$ and $R$ may jump together in that case).
The intertwining relation implies that the first coordinate of the process with generator $\hat{G}$ is an autonomous Markov process with generator $G$. 

Now, it may be justified that the generator $\hat{G}$ is the generator of $(R,L^1)$.
In fact, the process $L^1$ is plainly Markov in its own filtration and jumps from $\ell$ to $\ell +1$ at rate $c \ell (\ell-1)/2$. Then conditionally on the value of $L=\ell$, we view the $\ell$ first particles as $\ell$ immigrants,  and the process $R$ as a Wright Fisher diffusion with $\ell$ sources of immigration, $\ell-1$ sources of  type $2$ and one source of type $1$, whence the drift term $c \left[ (1-x)-(\ell-1)x \right]$ thanks to similar calculations as in \ref{K=K'}. 
We thus obtain the following pathwise decomposition of a Wright Fisher diffusion, which is another way to express the intertwining relation:
\begin{itemize}
\item Conditionally on $\{R_0=x\}$, the initial value $L^1(0)$ has law: $$\P(L^1(0)=\ell)=(1-x)^{\ell-1}x + \ind_{\{\infty\}}(\ell) \ind_{\{0 \}}(x), \ \ell \geq 1.$$ 
\item Conditionally on $(R_0, L^1(0))$, the process $L^1$ is a pure jump Markov process, which jumps from $\ell$ to $\ell+1$ at rate $c \ell(\ell-1)/2$ if $\ell <\infty$.
\item Conditionally on $(R_0, L^1)$, the process $R$ is a Wright Fisher diffusion with immmigration, with generator given by:
 $$\frac{1}{2} c x(1-x) f''(x) + \ind_{\{L^1<\infty\}} \ c  \left[  (1-x)-(L^1-1)x\right] f'(x).$$
\end{itemize}

\section{The additive $h$-transform}

In this Section, we derive another example of an $h$-transform (of measure valued processes) admitting a simple construction from the lookdown particle system. 

\subsection{The general construction of the lookdown particle system}
\label{spatial2}
%

We first present a more general construction of an exchangeable particle system, which allows to deal with type mutation and non constant population size. We recall this model was defined (in even greater generality) in \cite{DO99}.

Let $E$ be a Polish space. We consider a triple $(R_0,Y,U)$ constructed as follows. $R_0$ stands for a probability measure on $E$, and $Y=(Y_t, t \geq0)$ and $U=(U_t, t \geq 0)$ for two non negative real valued processes. We assume that $U_0=0$ and $U$ is non decreasing, so that $U$ admits a unique decomposition $U_t=U^k_t+ \sum_{s \leq t} \Delta U_s$ where $U^k$ is continuous (with Stieltjes measure denoted by $dU^k$) and $\Delta U_s=U_s-U_{s-}$.  We assume that $0$ is an absorbing point for $Y$, and set $\tau(Y)=\inf{\{t>0, Y_t=0}\}$ the extinction time of $Y$. We also assume that for each $s \geq 0$, $\Delta U_s \leq Y_s^2$.
Conditionally on $U$ and $Y$, we define two point measures $N^{\rho}$ and $N^k$ on $\R_{+}\times \p_{\infty}$, where $\p_{\infty}$ denotes the set of partition of $\N$: 
\begin{itemize}
 \item $N^{\rho}=\underset{0 \leq t < \tau(Y), \Delta U_t \neq 0} {\sum} \delta_{(t,\pi)}(dt,d\pi)$ where the exchangeable partitions $\pi$ of $\N$ are independent and have a unique non trivial block with asymptotic frequency $\sqrt{\Delta U_t}/Y_t$.
\item $N^{k}=\underset{0 \leq t < \tau(Y)} {\sum} \delta_{(t,\pi)}(dt,d\pi)$ is an independent Poisson point measure with intensity $(dU^k_t/(Y_t)^2)\times  \mu^{k}$, and the Kingman measure $\mu^k$ assigns mass one to partitions with a unique non trivial block consisting of two different integers, and mass $0$ to the others.
\end{itemize}
Conditionally on $(R_{0},Y,U)$, we then define a particle system $X=(X_t(n), 0 \leq t < \tau(Y), n \in \N)$ as follows:
\begin{itemize}
\item The initial state $(X_0(n), n \in \N)$ is an exchangeable sequence valued in $E$ with de Finetti's measure $R_{0}$.
\item At each atom $(t,\pi)$ of $N:=N^{k}+N^{\rho}$, we associate a \textit{reproduction event} as follows: let $j_1 < j_2< \ldots$ be the elements of the unique block of the partition $\pi$ which is not a singleton (either it is a doubleton if $(t,\pi)$ is an atom of $N^k$ or an infinite set if $(t,\pi)$ is an atom of $N^\rho$). The individuals $j_1<j_2 < \ldots$ at time $t$ are declared to be the children of the individual $j_1$ at time $t-$, and receive the type of the parent $j_1$, whereas the types of all the other individuals are shifted upwards accordingly, keeping the order they had before the birth event: for each integer $\ell$, $X_t(j_\ell)=X_{t-}(j_1)$ and for each $k \notin \{j_{\ell},\ell \in \N\}$, $X_t(k)=X_{t-}(k-\# J_k )$ with $J_k \equal \{\ell>1, j_{\ell} \leq k\}$ and $\# J_k$ the cardinal of the set $J_k$.
\item Between the reproduction events, the type $X_t(n)$ of the particle at level $n$ mutates according to a Markov process with c\`adl\`ag paths in $E$, with law $(\rP_{x}, x \in E)$ when started at $x \in E$, independently for each $n$.
\end{itemize}
This defines the particle system $X$ on $[0,\tau(Y))$. Since the process $X_s(j)$ admits a limit as $s$ goes to $\tau(Y)$ for each $j$, we may set $X_t(j)= \lim_{s \to \tau(Y)} X_s(j)$ for each $t>\tau(Y)$ and this sequence is still exchangeable according to Proposition 3.1 of \cite{DO99}. Conditionally on $(R_{0},Y,U)$, the sequence $(X_t(n), n \in \N))$ is well defined for each $t \in \R^+$ and exchangeable according to the same Proposition 3.1 of \cite{DO99}. We denote by $R_t$ its de Finetti measure:
\begin{equation*}
 R_t(dx)= \lim_{N \to \infty}   \frac{1}{N} \ \sum_{n=1}^{N} \delta_{X_t(n)}(dx),
\end{equation*}
and this defines a random process indexed by the set of non negative real numbers $\R^+$ since a c\`adl\`ag version of the process $R$ is shown to exist in \cite{DO99}. We finally define the measure valued process of interest $Z$ by:
\begin{equation}
\label{defZ}
(Z_t, t \geq 0)=(Y_t \; R_t,  t \geq 0). 
\end{equation}
The finite measure $Z$ represents the distribution of a population distributed in a space $E$, the process $Y$ corresponds to the total population size, and $U$ tracks the resampling inside the population.
We stress that, conditionally given $R_t$, the random variables $(X_t(n), n \in \N)$ on $E$ are independent and identically distributed according to the probability measure $R_t$ thanks to the de Finetti Theorem. 

We will denote by $\P$ the law of the triple $(Y,U, X)$. We introduce the relevant filtrations:
\begin{itemize}
\item $\left(\F_t=\sigma((Y_s, s\leq t),(X_s, s\leq t))\right)$ corresponds to the filtration of the particle system and the total population size.
\item $\left(\G_t=\sigma(Z_s, s\leq t)\right)$ corresponds to the filtration of the resulting measure valued process.
\item $\D_t$ is the filtration induced by the canonical process under $\rP$.
\end{itemize}
We shall use the classical notation $\mu(f) = \int \mu(dx) f(x)$ for a non negative map $f:E \to \R$ and $\mu \in \M_{f}$.
Note that $Y_t=Z_t(\ind)$, and thus $Y$ is $\G$-measurable.

\subsection{The additive $h$-transform}

We denote by $\M_f$ the space of finite measures on $E$. We call a non-negative function $H$ on $[0,\infty) \times \M_f$ a space-time harmonic function for $\P$ when the process $(H(t,Z_t), t \geq 0)$ is a martingale under $\P$. The $h$-transform $Z^H$ of $Z$ associated with $H$ is then defined by:
\begin{equation}
\label{htransform}
\forall A \in \G_t, \ 
\P(Z^H \in A) =\frac{H(t,Z_t)}{\E(H(0,Z_0))} \P \left(Z \in A\right).
\end{equation}
for every $t \geq 0$. Furthermore, an $h$-transform is called additive if there exists a non-negative function $(h_t(x), t\geq 0, x \in E)$ such that $H(t,Z_t)=Z_t(h_t)$.
\begin{rem}
Loosely speaking, an additive $h$-transform of the form \reff{htransform} favours the paths for which the population (represented by the measure valued process) is large where $h$ is large.
\end{rem}

\subsubsection{Statement of the results}
\label{results}

Let $\xi$ be the canonical process under $\rP_x$. We assume that $\left(Y_t/m(t), t \geq 0\right)$ and $(m(t)h_t(\xi_t), t \geq0)$ are martingales in their own filtrations for some positive deterministic function $m$.
We assume from now on that
\begin{equation*}
\E(Y_0 R_0(h_0))>0. 
\end{equation*}
Under this assumption, we define (the law of) a new process $$(Y^h,U^h,X^h)$$ by the following requirements:
\begin{itemize}
\item[(i)] The initial condition satisfies:
\begin{equation*}
\forall A \in \G_{t}, \P((Y^h_0,R^h_0) \in A)= \E\left( \frac{Y_0 R_0(h_0)}{\E(Y_0 R_0(h_0))} \ind_{A}(Y_0,R_0)\right).
\end{equation*}
\item[(ii)] Conditionally on $(Y^h_0,R^h_0)$, and provided $R^h_0(h_0)>0$, $X_{0}^h(1)$ is distributed according to:
\begin{equation*}
\forall A \in \D_0, \P(X_{0}^{h}(1) \in A| R_0^h=\mu)= \E \left( \frac{h_{0}(X_{0}(1))}{\mu(h_0)} \ind_{A}(X_0(1)) | R_0=\mu \right) ,
\end{equation*}
and $(X_{0}^h(n), n\geq 2)$ is an exchangeable random sequence with de Finetti's measure $R_{0}^h$.
\item[(iii)] Conditionally on $(Y^h_0,R^h_0, X_{0}^{h}(1))$, the process $(Y^h,U^h)$ is distributed according to:
\begin{equation}
\label{Yh}
\forall A \in \G_{t}, \ \P((Y^{h}, U^{h}) \in A| Y^h_0=x) = \E \left(\frac{Y_t}{x} \frac{m(0)}{m(t)} \; \ind_{A}(Y,U) | Y_0=x \right). 
\end{equation}
\item[(iv)] Conditionally on $(Y^h, U^h ,R^h_0, X_{0}^{h}(1))$, $X^{h}(1)$ is distributed according to:
\begin{equation}
\label{Xh}
\forall A \in \D_t, \; \P(X^h(1) \in A | X_0^h(1)=x)= \E \left( \frac{h_t(X_t(1))}{h_{0}(x)} \frac{m(t)}{m(0)} \ind_{A}(X(1))|X_0(1)=x\right). 
\end{equation}
\item[(v)] The rest of the definition of $X^h$ is the same as the one given for $X$, namely:
\begin{itemize}
\item for $n \geq 2$, between the reproduction events, the type $X^h_t(n)$ of the particle at level $n$ mutates according to a Markov process in $E$ with law $(\rP_{x}, x \in E)$ when started at $x \in E$, independently for each $n$.
\item at each atom $(t,\pi)$ of $N=N^{k}+N^{\rho}$, \textit{with $N^k$ and $N^\rho$ derived from $U^h$ and $Y^h$,} a reproduction event is associated as previously.
\end{itemize}
\end{itemize}
Note that the law of the initial condition $Z_0^h$ specified by (i) is different from that of $Z_0$ only for random $Z_0$. Also, notice that items (iii) and (iv) are meaningful since both $(Y_t/m(t), t \geq 0)$ and $(m(t) h_t(X_{t}(1)), t \geq 0)$ are assumed to be martingales. Last, we observe from \reff{Yh} that $\P(Y_t^h=0)=0$ for each $t \geq 0$, which implies $\P(\tau(Y^h)=\infty)=1$ since $0$ is assumed to be absorbing.
We will assume that $(Y,U,X)$ and $(Y^h,U^h,X^h)$ are defined on a common probability space with probability measure $\P$, and denote the expectation by $\E$.
\\
Let us define a process $S= (S_t, t \geq 0)$ by:
$$ S_t = \frac{h_t(X_t(1)) \; Y_t }{\E(Z_0(h_{0}))}.$$
\begin{lem}
\label{licit2}
The process $(S=S_t, t \geq 0)$ is a non negative $\F$-martingale, and 
\begin{equation}
\forall A \in \F_t, \ \P(X^h \in A)= \E \left( \ind_{A}(X) \ S_t \right).
\end{equation}
\end{lem}
\noindent We then define the process $T$:
$$T_t= \frac{ Z_t(h_t)}{\E(Z_0(h_0))}.$$
Using Lemma \ref{licit2}, and projecting on the filtration $\G_t$, we deduce Lemma \ref{licit}.
\begin{lem}
\label{licit}
The process $T=(T_t, t\geq 0)$ is a non negative $\G$-martingale.
\end{lem}
\noindent This fact allows to define the process $Z^H\equal (Z^H_t, t \geq 0)$ absolutely continuous with respect to $Z\equal (Z_t, t \geq 0)$ on each $\G_t$, $t \geq 0$, with Radon Nykodim derivative:
\begin{equation*}
\forall A \in \G_t, \ \P(Z^H \in A)= \E\left(\ind_{A}(Z) \ T_t \right).
\end{equation*}
We deduce from Lemma \ref{licit2} and Lemma \ref{licit} the following Theorem.
\begin{theo}
\label{additive}
We have that:
\begin{itemize}
 \item[(a)] The limit of the empirical measure:
\begin{equation*}
R^h_t(dx) \equal \lim_{N\to \infty} \frac{1}{N} \sum_{n=1}^{N} \delta_{X_{t}^h(n)}(dx)
\end{equation*}
exists a.s.
\item[(b)] The process $(Z^h_t \equal Y^h_t R^h_t, t \geq 0)$ is distributed as $(Z^H_t, t \geq 0)$.
\end{itemize}
\end{theo}

\begin{rem}
\label{interpretation}
 We may interpret Theorem \ref{additive} as follows. The effect of the additive $h$-transform factorizes in two parts, according to the decomposition of the Radon Nikodym derivative: 
$$Z_t(h_t)=Y_t \;  R_t(h_t).$$
The first term $Y_t$ induces a size bias of the total population size $Z^h(\ind)=Y^h$, see formula \reff{Yh}, whereas the second term $R_t(h_t)$ forces the first level particle to follow an $h$-transform of $\rP$, see formula \reff{Xh}. 
\end{rem}

The sequence $(X_t^h(n), n\in \N)$ is not exchangeable in general, which contrasts with the initial sequence $(X_t(n), n\in \N)$. The following Proposition shows that, loosely speaking, the first level particle is precursory.
\begin{prop}
\label{distinguish}
Conditionally on $\{ R^h_t=\mu\}$,  $X_{1}^h(t)$ is distributed according to:
\begin{equation*}
\P(X_{t}^{h}(1) \in dx)= \frac{h_t(x)}{\mu(h_t)} \mu(dx),
\end{equation*}
and $(X_{t}^h(n))_{n \geq 2}$ is an independent exchangeable random sequence with de Finetti's measure $\mu$.
\end{prop}

\subsubsection{Proofs}

\begin{proof}[Proof of Lemma \ref{licit2}]
It is enough to observe that, by construction, the law of $(Y^h,U^h,X^h)$ is absolutely continuous with respect to the law of $(Y, U, X)$ on $\F_{t}$, with Radon Nykodim derivative given by:
\begin{align*}
\forall A \in \F_{t}, \ \P((Y^{h},U^h,  X^h) \in A)
&= \E\left( \frac{Y_0 R_0(h_0)}{\E(Y_0 R_0(h_0))} \frac{h_0(X_0^h(1))}{R_{0}(h_{0})} \frac{Y_t}{Y_0} \frac{m(0)}{m(t)} \frac{h_t(X_t(1))}{h_0(X_0^h(1))} \frac{m(t)}{m(0)} \ind_{A}(Y,U,X) \right)\\ 
&=\E \left( \frac{Y_t \;  h_t(X_t(1)) }{\E(Z_0(h_{0})) } \ind_{A}(Y,U,X) \right) 
\end{align*}
This also yields (the obvious fact) that $\left(S_t, t \geq 0 \right)$ is a $\F$-martingale.
\end{proof}

\begin{proof}[Proof of Lemma \ref{licit}]
From Lemma \ref{martingale_mult}, since $\G_t  \subset \F_t$ and $S$ is a $\F$-martingale, the projection $\E(S_t|\G_t)$ is a $\G$-martingale. We also have:
\begin{align*}
\E(S_t|\G_t) = \E\left( \frac{Y_t \;  h_t(X_t(1)) }{\E(Z_0(h_{0}))}| \G_t \right) = \frac{Z_t(h_t)}{\E(Z_0(h_{0}))} =T_t,
\end{align*}
where we used that $X_{t}(1)$ has law $R_t$ conditionally on $\G_t$ for the third equality. Thus $(T_t, t \geq 0)$ is a $\G$-martingale.
\end{proof}

\begin{proof}[Proof of Theorem \ref{additive}]

From Lemma \ref{licit2}, the law of $X^h$ is absolutely continuous with respect to the law of $X$.  The existence of the a.s. limit of the empirical measure of $X^h$ follows from that of $X$ (but not the exchangeability of the sequence) and yields point (a). 
We prove point (b) now. Take $A \in \G_t$. 
\begin{align*}
\P(Z^{h} \in A) &= \E \left(S_t\ind_{A}(Z) \right) \\
&= \E \left(  \E \left( S_t | \G_t \right) \ind_{A}(Z) \right) \\
&= \P( T_t \; \ind_{A}(Z)) \\
&= \P( Z^H \in A),
\end{align*}
where we use Lemma \ref{licit} at the third equality and the definition of $Z^H$ for the last equality.  

\end{proof}

\begin{proof}[Proof of Proposition \ref{distinguish}]
Let $n\in \N$ be fixed, and let $(\phi_i)_{(1 \leq i \leq n)}$ be a collection of bounded and measurable functions on $E$.
\begin{align*}
\E\left(\prod_{1 \leq i \leq n} \phi_i(X_i^h(t))\right) &= \E \left( \frac{Y_t  \; h_t(X_{t}(1))}{\E(Z_0(h_{0}))}\prod_{1 \leq i \leq n} \phi_i(X_t(i)) \right) \\
&= \frac{1}{\E(Z_0(h_{0}))} \E \left( Y_t \; \E\left(h_t(X_{t}(1)) \phi_1(X_{t}(1)) \prod_{2 \leq i \leq n} \phi_i(X_t(i)) |\G_t \right) \right) \\
&= \frac{1}{\E(Z_0(h_{0}))} \E \left(Y_t \;  R_t(h_t \; \phi_1) \prod_{2 \leq i \leq n} R_t(\phi_i) \right) \\
&=  \E \left( \frac{Z_t(h_t)}{\E(Z_0(h_{0}))}  \; R_t\left(\frac{h_t \;  \phi_1}{R_t(h_t)}\right) \prod_{2 \leq i \leq n} R_t(\phi_i) \right) \\
&=  \E \left(R^h_t\bigg(\frac{h_t \; \phi_1}{R^h_t(h_t)}\bigg) \prod_{2 \leq i \leq n} R^h_t(\phi_i) \right), \\
\end{align*}
where we use Lemma \ref{licit2} at the first equality, the de Finetti Theorem at the third equality, and Theorem \ref{additive} at the last equality. Since functions of the type $\prod_{1 \leq i \leq n} \phi_i$ characterize the law of $n$-uple, this proves the Proposition. 
\end{proof}

\subsection{Applications}

Overbeck investigated in \cite{DA92} $h$-transform of measure valued diffusions, among which the Dawson Watanabe process (with quadratic branching mechanism) and the Fleming Viot process (which is the GFV process for $\nu=0$) using a martingale problem approach. He also provided a pathwise constructions in the first case, see \cite{OV94}.
We shall see in this last Section how Theorem \ref{additive} applies in both cases and sheds new light on Overbeck's results.

\subsubsection{Generalized Fleming Viot processes}
\label{subsectionGFV}

Recall the Generalized Fleming Viot process (with mutation) is the process $Z$ constructed in Section \ref{spatial2} when setting:
\begin{itemize}
\item $Y=\ind$,
\item $U$ is a subordinator with jumps no greater than $1$,
\end{itemize}
It also corresponds to the process $R$ introduced in Section \ref{spatial} when allowing for mutations.
We denote by $\phi$ the Laplace exponent of the subordinator $U$:
\begin{equation*}
\phi(\lambda)= c \lambda+ \int_{(0,1]} (1- \expp{-\lambda x}) \nu^U(dx)
\end{equation*}
where $c \geq 0$ and the L\'evy measure $\nu^U$ satisfies $\int_{(0,1]} x \; \nu^U(dx) < \infty$. The genealogy of the lookdown particle system is by construction described by the $\Lambda$-coalescent of Pitman \cite{PI99}. The finite measure $\Lambda$ is related to  $\nu$ and $c$  by $ x^{-2} \Lambda_{|(0,1]}(dx)=\nu(dx)$ and $\Lambda\{0\}=c$, and may be recovered from the characteristics of $\phi$ through the identity:
\begin{equation*}
\int_{[0,1]} g(x) \Lambda(dx) =c g(0)+ \int_{(0,1]} g(\sqrt{x}) \nu^U(dx), 
\end{equation*}
see Section 3.1.4 of \cite{DO99}.

Since $Y_t=1$, $Y$ is a martingale and we may apply results of Section \ref{results} for any non-negative space time harmonic function $(h_t(x), t \geq 0, x \in E)$ for the spatial motion $\rP$, that is any function such that $(h_t(\xi_t), t \geq 0)$ is a non-negative martingale where $\xi$ stands for the canonical process under $\rP$. Notice the construction of the particle system $X^h$ simplifies here since $(U^h,Y^h)\overset{(law)}{=}(U,Y)=(U,\ind)$.

Overbeck suggested in \cite{DA92} that in the particular case of the FV process, an additive $h$-transform looks like a FV process where ``\textit{the gene type of at least one family mutates as an $h$-transform of the one particle motion}''. This suggestion was made ``\textit{plausible}'' by  similar results known for superprocesses, see \cite{OV94} or the next Subsection, and a well known connection between superprocesses and Fleming-Viot processes which goes back to Shiga \cite{SH90}. We did not attempt to derive the additive $h$-transform for GFV processes in this way, since the connection between superprocesses and GFV processes is restricted to stable superprocesses and Beta GFV processes, see Birkner \textit{et. al} \cite{BI05}.

Theorem \ref{additive} allows us to see at first glance that the family which ``\textit{mutates as an $h$-transform}'' is the family generated by the first level particle in the lookdown process. Other advantages of Theorem \ref{additive} over the martingale problem approach is that it provides a pathwise approach and also applies for GFV processes.

\begin{rem}
\label{discretize}
If $\nu^U=0$ (or, equivalently, $\Lambda(dx)=\Lambda\{0\} \;\delta_{0}$), the truncated processes obtained by considering the first $N$ particles:
\begin{equation*}
Z^N_t(dx) \equal \frac{1}{N} \sum_{1 \leq n \leq N} \delta_{X_t(n)}(dx) \mbox{ and } Z^{N,h}_t(dx) \equal \frac{1}{N} \sum_{1 \leq n \leq N} \delta_{X^h_t(n)}(dx)
\end{equation*}
correspond respectively to the Moran model with $N$ particles (see \cite{DO99}) and its additive $h$-transform.
This proves our approach is robust, in the sense that we can also consider discrete population.
\end{rem}

Finally, we may interpret the $h$-transform as a conditioned process.
For fixed $s \geq 0$, the additive $h$-transform of the GFV process on $[0,s]$ may be obtained by conditioning a random particle chosen at time $t$, $t$ large, to move as an h-transform. For other conditionings on boundary statistics in the context of measure valued branching processes, we refer to Salisbury and Sezer \cite{SS09}.

\subsubsection{The Dawson Watanabe superprocess}
\label{DawsonWatanabe}

Recall a continuous state branching process is a strong Markov process characterized by a branching mechanism $\psi$ taking the form
\begin{align}
\label{branchingmechanism}
\psi(\lambda)&=\frac{1}{2}\sigma^{2}\lambda^{2} + \beta \lambda +\int_{(0,\infty)}(\expp{-\lambda u}-1+\lambda u\ind_{u\leq 1})\nu^Y(du),
\end{align}
for $\nu^Y$ a L\'evy measure such that $\int_{(0,\infty)} (1\wedge u^2) \nu^Y(du)<\infty$, $\beta \in \R$, and  $\sigma^2 \in \R^+$. We will denote it  CB($\psi$) for short.
More precisely, the CB($\psi$) process is the strong Markov process $Y$ with Laplace transform given by:
\begin{align*}
\E(\expp{- \lambda Y_t}|  Y_0=x)=\expp{-x u(\lambda,t)},
\end{align*}
where $u$ is the unique non-negative solution of the integral equation, holding for all $t \geq 0$, $\lambda \geq 0$: 
\begin{equation}
\label{u}
u(\lambda,t)+\int_{0}^{t} ds \  \psi\left(u(\lambda,s)\right)= \lambda.
\end{equation}
We assume that $\psi'(0+)>-\infty$, so that the CB($\psi$) has integrable marginals, and
\begin{equation*}
 (Y_t \expp{\psi'(0+) t}, t \geq 0)
\end{equation*}
is a martingale. The Dawson Watanabe process with general branching mechanism given by $\psi$ is the measure valued process $(Z_t, t \geq 0)$ constructed in Section \ref{spatial2} when:
\begin{itemize}
 \item $(Y_t, t \geq 0)$ is a CB.
\item $(U_t, t \geq 0)$ is the quadratic variation process of $Y$, $U_t=[Y](t)$. Therefore, $\Delta U_t=(\Delta Y_t)^2 \leq Y_t^2$.
\end{itemize}
Since the process $(Y_t \expp{\psi'(0+) t}, t \geq 0)$ is a martingale, we may apply our results for any non-negative function $(h_t(x), t \geq 0, x \in E)$ such that $(h_t(\xi_t) \expp{-\psi'(0+) t},t \geq 0)$ is a martingale.

We now link our results with the literature:
\begin{enumerate}
 \item When $Y$ is a subcritical CB process, meaning that $\psi'(0+) \geq0$, setting $m(t)=\expp{-\psi'(0+) t}$ and $h_t(x)=\expp{ \psi'(0+) t}$, we recover from Theorem \ref{additive} part of the Roelly \& Rouault \cite{RR89} and Evans \cite{EV93} decomposition. This $h$-transform may be interpreted in that case as the process conditioned on non extinction.
\item When $Y$ is a critical Feller diffusion and $\rP$ the law of a Brownian motion, if we assume we are given $h_t(x)$ a space time harmonic function for $\rP$ and set $m(t)=1$, we get from Theorem \ref{additive} the decomposition of the $h$-transform of the Dawson Watanabe process provided by Overbeck in \cite{OV94}.
\end{enumerate}

We identified in Remark \ref{interpretation} two effects of the additive $h$-transform: the total population is size biased and the first level particle follows an $h$-transform of $\rP$. We shall concentrate on the first effect, and explain how a ``spinal'' decomposition may be partly recovered from Theorem \ref{additive}: Lemma \ref{CBI-bias} identifies the size biased total mass process $Y^h=Z^h(\ind)$ with a branching process with immigration, and Lemma \ref{interpret} recognizes the first level particle as the source of the immigration.


Let $\phi$ be the Laplace exponent of a subordinator. Recall a continuous state branching process with immigration with branching mechanism $\psi$ and immigration mechanism $\phi$, CBI($\psi$,$\phi$) for short, is a strong Markov process $(Y^i_t, t \geq 0)$ characterized by the Laplace transform:
\begin{equation*}
\E(\expp{- \lambda Y^i_t}|  Y^i_0=x)=\expp{-x u(\lambda,t) - \int_{0}^{t} ds \ \phi(u(\lambda,s))} .
\end{equation*}
We recall for the ease of reference the following well known lemma, and we stress that this Lemma also holds in the supercritical case $\psi'(0+) <0$.
\begin{lem}
\label{CBI-bias}
The process $Y^h$ defined by \reff{Yh} with $m(t)=\expp{-\psi'(0+) t}$ is a CBI($\psi$,$\phi$) with immigration mechanism given by $\tilde{\phi}(\lambda) \equal \psi'(\lambda)-\psi'(0+)$. 
\end{lem}
The proof is classical and relies on computation of the Laplace transforms. Notice that in the case where the CB process $Y$ extincts almost surely, the CBI  process $Y^h$ may also be interpreted as the CB process $Y$ conditioned on non extinction in remote time, see Lambert \cite{L07}.

The total mass process $Y^h=Z^h(\ind)$ is from Lemma \ref{CBI-bias} a CBI process whereas $Y$ is by assumption a CB process. We may thus wonder ``who'' are the immigrants in the population represented by the particle system $X^h$. The following Lemma shows that the offsprings of the first level particle are the immigrants.  Recall $j_1$ refers to the first level sampled in the lookdown construction. Let us denote $j_1(s)$ instead of $j_1$ for indicating the dependence in $s$.
\begin{lem}
\label{interpret}
The process $\left( \sum_{0 \leq s \leq t} \Delta Y^h_s \; \ind_{\{j_1(s)=1\}}, t \geq 0 \right)$ is a pure jump subordinator with L\'evy measure $u\nu^Y(du)$. 
\end{lem}


\begin{proof}
By assumption, the process $Y$ is a CB($\psi$) and from Lemma \ref{CBI-bias}, $Y^h$ is a CBI($\psi,\tilde{\phi}$).
From the Poissonian construction of CBI, we have that the point measure 
$$\sum_{0 \leq s \leq t} \delta_{(s,\Delta Y^h_s)} (ds,du)$$ 
has for predictable compensator
$$ ds \ (Y^h_{s-} \nu^Y(du) + u \nu^Y(du)).$$
The expression of the compensator may be explained as follows. The term $ds \ Y^h_{s-} \nu^Y(du)$ comes from the time change of the underlying spectrally positive L\'evy process, called the Lamperti time change (for CBs). The term $ds \ u \nu^Y(du)$ is independent of the current state of the population and corresponds to the immigration term. 
Then, conditionally on the value of the jump $\Delta Y^h_s=u$, the event $\{j_1(s)=1\}$ has probability $$\frac{u}{Y^h_s}=\frac{u}{Y^h_{s-}+u}$$ independently for each jump. Therefore, the predictable compensator of the point measure $$\sum_{0 \leq s \leq t} \delta_{(s,\Delta Y^h_s)} (ds,du) \ind_{\{j_1(s)=1\}}$$ is
$$ ds \ \left( \frac{u}{Y^h_{s-}+u} \right) (Y^h_{s-} \nu^Y(du) + u \nu^Y(du))= ds \  u \nu^Y(du), $$ 
This ends up the proof.
\end{proof}

\begin{rem}
Understanding the action of the continuous part of the subordinator requires to work with the discrete particle system generated by the first $N$ particles. Namely, it is possible to prove that the family of processes
$$ \left(\sum_{0 \leq s \leq t} Y_s^h \;  \frac{\#{\{1 \leq i \leq N, j_i(s) \leq N \}}}{N} \ind_{\{j_1(s)=1, j_2(s) \leq N\}}, t\geq 0 \right)$$
converges almost surely as $N\to \infty$ in the Skorohod topology towards a subordinator with Laplace exponent $\tilde{\phi}$.
%
\end{rem}

\textbf{Acknowledgments}. The author is grateful to Stephan Gufler and Anton Wakolbinger for letting him know about the intertwining relationship found by Jan Swart and to Jean-Fran\c  cois Delmas for careful reading.


\end{document}